\documentclass{article}

\usepackage{amssymb}
\usepackage{amsmath}
\usepackage{amsthm}
\usepackage{cleveref}
\crefname{equation}{}{}
\Crefname{equation}{Equation}{Equations}
\crefname{theorem}{Theorem}{Theorems}
\Crefname{theorem}{Theorem}{Theorems}
\crefname{lemma}{Lemma}{Lemmas}
\Crefname{lemma}{Lemma}{Lemmas}
\crefname{proposition}{Proposition}{Propositions}
\Crefname{proposition}{Proposition}{Propositions}
\crefname{corollary}{Corollary}{Corollaries}
\Crefname{corollary}{Corollary}{Corollaries}
\crefname{conjecture}{Conjecture}{Conjectures}
\Crefname{conjecture}{Conjecture}{Conjectures}
\crefname{section}{Section}{Sections}
\Crefname{section}{Section}{Sections}
\crefname{example}{Example}{Examples}
\Crefname{example}{Example}{Examples}

\newtheorem{theorem}{Theorem}
\newtheorem{lemma}[theorem]{Lemma}

\newcommand{\Z}{\mathbb{Z}}

\title{On Tur\'{a}n problems \\
for Cartesian products of graphs}
\author{Alexander Sidorenko}
\date{\today}

\begin{document}

\maketitle

\begin{abstract}
Let $A,B$ be disjoint sets of sizes $n$ and $m$. 
Let ${\mathcal Q}$ be a family of quadruples, 
having $2$ elements from $A$ and $2$ from $B$, 
such that any subset $S \subseteq A \cup B$ 
with $|S|=7$, $|S \cap A| \geq 2$  and $|S \cap B| \geq 2$ 
contains one of the quadruples. 
We prove that the smallest size of ${\mathcal Q}$ is 
$(1/16 + O(1/n) + O(1/m)) n^2 m^2$ as $n,m\to\infty$. 
We also solve asymptotically a more general 
two-partite Tur\'{a}n problem for quadruples.
\end{abstract}


An $r$-{\it graph} is a pair $H=(V(H),E(H))$ 
where $V(H)$ is a finite set of vertices, 
and the edge set $E(H)$ is a collection of $r$-subsets of $V(H)$. 
A subset of vertices is called {\it independent} 
if it contains no edges of $H$. 
The {\it independence number} $\alpha(H)$ 
is the largest size of an independent subset. 

The classical Tur\'{a}n number $T(n,k,r)$ 
is the minimum number of edges 
in an $n$-vertex $r$-graph $H$ with $\alpha(H) < k$. 
Consequently, $\binom{n}{r} - T(n,k,r)$ 
is the largest number of edges in an $n$-vertex $r$-graph 
that does not contain a complete subgraph on $k$ vertices.
The ratio $T(n,k,r) / \binom{n}{r}$ is non-decreasing 
when $n$ increases, so the limit 
$t(k,r) = \lim_{n\to\infty} T(n,k,r) / \binom{n}{r}$ 
exists. 

The exact values of $T(n,k,2)$ were found 
by Mantel \cite{Mantel:1907} for $k=3$, 
and by Tur\'{a}n \cite{Turan:1941} for all $k$: 
\[
  T(n,k,2) \; = \;
  mn - \frac{m(m+1)}{2} \: (k-1) 
  \;\;\;\;\;{\rm if}\;\; m \leq \frac{n}{k-1} \leq m+1 \; .
\]
In particular, $t(k,2) = 1/(k-1)$. 
Not a single value $t(k,r)$ is known with $k>r>2$. 

Giraud \cite{Giraud:1990} 
discovered an elegant construction for $r=4$, $k=5$ which yields 
$t(5,4) \leq 5/16$. 
This construction was generalized 
by de Caen, Kreher and Wiseman \cite{Caen:1988} 
in the following way. 
Consider two disjoint sets 
$A=\{a_1,a_2,\ldots,a_n\}$, $B=\{b_1,b_2,\ldots,b_m\}$, 
and a $0,1$-matrix $X=[x_{ij}]$ of size $n \times m$. 
Let $E_{40}$ be the set of all quadruples within $A$, 
$E_{04}$ be the set of all quadruples within $B$, 
and $E_{22}$ be the set of quadruples $\{a_i,a_j,b_p,b_q\}$ 
such that $x_{ip} + x_{iq} + x_{jp} + x_{jq}$ is even.
It is easy to see that in the $4$-graph 
$H=(A \cup B,\: E_{40} \cup E_{04} \cup E_{22})$ 
any subset of $5$ vertices contains at least one edge. 
If $n$ and $m$ are approximately equal,
and the entries of $X$ are selected randomly and independently, 
with equal probability of being $0$ or $1$, 
the expected number of edges in $H$ is 
$\frac{5}{16} \binom{n+m}{4} + O((n+m)^3)$. 
A more specific choice of $X$ related to Hadamard matrices 
provides the best known upper bounds for $T(n,5,4)$ 
(see \cite{Sidorenko:1995}). 

If $n,m\to\infty$, then
$\big( \frac{1}{2} + o(1) \big) \binom{n}{2}\binom{m}{2}$ 
is the minimal size of a system of quadruples 
with $2$ elements from $A$ and from $B$ 
such that every quintuple 
with $3$ elements from $A$ and $2$ from $B$, or vice versa, 
contains at least one of the quadruples.

In \cite{Caen:1990}, de Caen, Kreher and Wiseman 
defined a broader set of Tur\'{a}n type problems 
which we describe here for quadruples setting only. 
Let $Q(n,m,a,b)$ denote the smallest size of a system of quadruples 
with $2$ elements from $A$ and from $B$ 
such that every $(a+b)$-set 
with $a$ elements from $A$ and $b$ from $B$ 
contains at least one of the quadruples. 
Obviously, $Q(n,m,a,b) \leq T(n,a,2) \cdot T(m,b,2)$. 
It was proved in \cite{Caen:1990} that 
$Q(n,n,3,3) = q \binom{n}{2}^2 + o(n^4)$ 
where $1/4 \geq q \geq (3-\sqrt{5})/4 \approx 0.1910$. 

In this note, we consider the following problem. 
Let ${\mathcal P}$ be a set of pairs $(a,b)$ with $a,b\geq 2$. 
Let $A$ and $B$ be disjoint sets with $|A|=n$ and $|B|=m$. 
Denote by $Q(n,m,{\mathcal P})$ the minimum size of a system of quadruples 
with $2$ elements from $A$ and from $B$ 
such that for every $(a,b)\in{\mathcal P}$, 
every $(a+b)$-set with $a$ elements from $A$ and $b$ from $B$ 
contains at least one of the quadruples. 
The cases 
${\mathcal P} = \{(2,3),(3,2)\}$ and ${\mathcal P} = \{(3,3)\}$ 
were studied in \cite{Caen:1988} and \cite{Caen:1990}, respectively.
We will focus on the cases 
${\mathcal P}_{k2} = \{(2,k+1),(k+1,2)\}$ and
${\mathcal P}_{k3} = \{(2,k+1),(3,k),(k,3),(k+1,2)\}$. 
To solve them, we need to generalize the above-mentioned construction 
with $2 \times 2$ submatrices.

\vspace{3mm}
\noindent
{\bf Definition.} 
A $2 \times 2$ matrix $[x_{ij}]$ over an additive abelian group 
is called {\it fair} if $x_{11} + x_{22} = x_{12} + x_{21}$. 

\begin{lemma}\label{th:fair2}
Every $2 \times (k+1)$ matrix $[x_{ij}]$ over $\Z_k$
contains a fair $2 \times 2$ submatrix. 
\end{lemma}

\begin{proof}[{\bf Proof}]
Among the $k+1$ values $x_{1i} - x_{2i}$ one can find two equal ones. 
If $x_{1i} - x_{2i} = x_{1j} - x_{2j}$, 
then columns $i$ and $j$ form a fair submatrix.
\end{proof}

\begin{lemma}\label{th:fair3}
If $k$ is even, every $3 \times k$ matrix $X=[x_{ij}]$ over $\Z_k$ 
contains a fair $2 \times 2$ submatrix.
\end{lemma}

\begin{proof}[{\bf Proof}]
It is obvious that a fair submatrix remain fair 
after adding the same row to every row of $X$. 
Since one may subtract the first row from each of the three rows, 
we can assume that $x_{11}=x_{12}=\ldots=x_{1k}=0$. 
If there are two equal entries in the second or in the third row, 
then $X$ has a fair submatrix. 
Hence, we can assume that both the second and the third row contain 
every element of $\Z_k$ exactly once. 
Let $S_n$ be the sum of entries in row $n$, 
then $S_2 = S_3$. 
If there are two columns $i$ and $j$ such that 
$x_{2i} - x_{3i} = x_{2j} - x_{3j}$, 
then $X$ has a fair submatrix in columns $i,j$ and rows $2,3$. 
Hence, we can assume that $k$ values $x_{1i} - x_{2i}$
represent the $k$ distinct elements of $\Z_k$, but then
$S_2 - S_3 = \sum_{i=1}^k (x_{2i} - x_{3i}) = 0+1+\ldots+(k-1) \equiv k/2 \pmod{k}$, 
a contradiction.
\end{proof}

Let $G_k$ be a graph whose vertices are functions $f: \Z_k\to\Z_k\:$. 
A pair of vertices $\{f,g\}$ forms an edge in $G_k$ 
if $f-g$ is a bijection. 
\Cref{th:fair3} restates the fact 
that $G_k$ has no triangles when $k$ is even. 
For odd $k$, the problem of counting triangles in $G_k$ 
has been solved asymptotically in \cite{Eberhard:2015}. 
Let $p(k)$ be the smallest prime factor of $k$. 
The $p(k)$ functions $f_0,f_1,\ldots,f_{p(k)-1}$, 
where $f_i(j) = i \cdot j \pmod{k}$, 
form a complete subgraph in $G_k$. 
It is very tempting to conjecture that $p(k)$ 
is indeed the size of the largest clique in $G_k$. 
We know that this is true for even $k$ and for prime $k$.
Computer search confirms that this is also true for $k=9$.

\begin{theorem}\label{th:main}
$
 Q(n,m,{\mathcal P}_{k2}) = 
 \big( \frac{1}{4k} + O(\frac{1}{n}) + O(\frac{1}{m}) \big) n^2 m^2
$,
and if $k$ is even, 
$
 Q(n,m,{\mathcal P}_{k3}) = 
 \big( \frac{1}{4k} + O(\frac{1}{n}) + O(\frac{1}{m}) \big) n^2 m^2
$
as $n,m\to\infty$.
\end{theorem}

\begin{proof}[{\bf Proof}]
Let $A$ and $B$ be disjoint sets of sizes $n$ and $m$. 
Let ${\mathcal Q}$ be a system of quadruples such that 
every $(k+3)$-set with $2$ elements from $A$ and $k+1$ from $B$ 
contains at least one of the quadruples. 
Obviously, for each pair $\{u,v\}\subseteq A$, 
the number of quadruples in ${\mathcal Q}$ 
that contain $\{u,v\}$ is at least $T(m,k+1,2)$, 
hence $Q(n,m,{\mathcal P}_{k2}) \geq \binom{n}{2} T(m,k+1,2)$, 
and similarly, $Q(n,m,{\mathcal P}_{k2}) \geq \binom{m}{2} T(n,k+1,2)$, 
which yields
$
  Q(n,m,{\mathcal P}_{k3}) \geq Q(n,m,{\mathcal P}_{k2}) \geq 
  \big( 1/(4k) + O(1/n) + O(1/m) \big) n^2 m^2
$
as $n,m\to\infty$.

To prove the upper bound, 
consider an $n \times m$ matrix $X=[x_{ij}]$ over $\Z_k$. 
Let $A=\{a_1,a_2,\ldots,a_n\}$, $B=\{b_1,b_2,\ldots,b_m\}$, 
and let ${\mathcal Q}$ consist of quadruples $\{a_i,a_j,b_p,b_q\}$ 
such that rows $i,j$ and columns $p,q$ in $X$ 
produce a fair $2 \times 2$ submatrix. 
By \cref{th:fair2}, 
$Q(n,m,{\mathcal P}_{k2}) \leq |{\mathcal Q}|$, 
and if $k$ is even, by \cref{th:fair3}, 
$Q(n,m,{\mathcal P}_{k3}) \leq |{\mathcal Q}|$. 
If entries of $X$ are selected randomly, 
independently and uniformly over $\Z_k$,  
the expected value of $|{\mathcal Q}|$ is 
$\frac{1}{k}\binom{n}{2}\binom{m}{2}$ 
which provides the required upper bound.
\end{proof}

The result mentioned in the abstract follows from \cref{th:main} 
when $k=4$ and ${\mathcal P} = {\mathcal P}_{43}$.

Tur\'{a}n problems, 
where the extremal configurations depend on random maps
defined on the set of pairs of vertices, 
have been studied in a recent series of articles 
by R\"{o}dl and his coauthors 
(see the concluding remarks in \cite{Reiher:2018}).
Another ``partite'' version of Tur\'{a}n problem 
and its connection to the classical problem 
has been studied by Talbot \cite{Talbot:2007}.

\section*{Acknowledgments}

The author would like to thank two anonymous referees 
for their careful reading and valuable suggestions.

\end{document}